\documentclass[amstex,12pt, amssymb]{article}

\usepackage{mathtext}
\usepackage[cp1251]{inputenc}
\usepackage[T2A]{fontenc}
\usepackage[dvips]{graphicx}
\usepackage{amsmath}
\usepackage{amssymb}
\usepackage{amsxtra}
\usepackage{latexsym}
\usepackage{ifthen}

\newcounter{lemma}[section]

\newcounter{corollary}[section]

\newcounter{remark}[section]

\newcounter{theorem}[section]

\newcounter{proposition}[section]

\numberwithin{equation}{section}

\def\XXint#1#2#3{{\setbox0=\hbox{$#1{#2#3}{\int}$}
     \vcenter{\hbox{$#2#3$}}\kern-.5\wd0}}

\def\cc{\setcounter{equation}{0}
\setcounter{figure}{0}\setcounter{table}{0}}

\begin{document}

\markboth{\centerline{V. Gutlyanskii, V. Ryazanov, E. Yakubov, A.
Yefimushkin}} {\centerline{On Hilbert problem for Beltrami equation \\
in quasihyperbolic domains}}

\author{{V.Gutlyanskii, V.Ryazanov, E. Yakubov, A.Yefimushkin}}

\title{{\bf On Hilbert problem for Beltrami equation \\
in quasihyperbolic domains}}

\maketitle

\large \begin{abstract} We study the Hilbert boundary value problem
for the  Beltrami equation in the Jordan domains satisfying the
quasihyperbolic boundary condition by Gehring--Martio, generally
speaking, without $(A)-$condition by Ladyzhenskaya--Ural'tseva that
was standard for boundary value problems in the PDE theory. Assuming
that the coefficients of the problem are functions of countable
bounded variation and the boundary data are measurable with respect
to the logarithmic capacity, we prove the existence of the
generalized regular solutions. As a consequence, we derive the
existence of nonclassical solutions of the Dirichlet, Neumann and
Poincare boundary value problems for generalizations of the Laplace
equation in anisotropic and inhomogeneous media.
\end{abstract}



\bigskip

{\bf In  memory of Professor Bogdan Bojarski}

\large \cc
\section{Introduction}

D. Hilbert \cite{H1} studied the boundary value problem formulated
as follows: To find an analytic function $f(z)$ in a domain $D$
bounded by a rectifiable Jordan contour $C$ that satisfies the
boundary condition
\begin{equation}\label{2}
\lim\limits_{z\to\zeta}\ {\rm{Re}}\,
\{\overline{\lambda(\zeta)}\
f(z)\}\ =\ \varphi(\zeta) \quad\quad\quad\ \ \ \forall \ \zeta\in C\
,
\end{equation}
where both the {\it coefficient} $\lambda$ and the {\it boundary
date} $\varphi$ of the problem are con\-ti\-nu\-ous\-ly
differentiable with respect to the natural parameter $s$ on $C.$

Moreover, it was assumed by Hilbert that $\lambda\ne 0$ everywhere
on $C$. The latter allows us, without loss of generality, to
consider that $|\lambda|\equiv 1$ on $C$. Note that the quantity
${\rm{Re}}\,\{\overline{\lambda}\, f\}$ in (\ref{2}) means a
projection of $f$ into the direction $\lambda$ interpreted as
vectors in $\mathbb R^2$.

\medskip

The reader can find a rather comprehensive treatment of the theory
in the new excellent books \cite{Be,BW,HKM,TO}. We also recommend
to make familiar with the historic surveys contained in the
monographs \cite{G,Mus,Vek} on the topic with an exhaustive
bibliography and take a look at our recent  papers
\cite{GR,GRY,R1}.

\medskip

In this paper we study the Hilbert boundary value problem in a wider
class of functions than those of analytic. Namely, instead of
analytic functions we will consider {\it quasiconformal functions}
$F$ represented as a composition of analytic functions ${\cal A}$
and quasiconformal mappings $f$, see \cite{LV}, Chapter VI. In this
connection, we need to recall some definitions and notations from
the theory of quasiconformal mappings in the plane.

\medskip

Let $D$ be a domain in the complex plane $\mathbb C$ and let $\mu:
D\to\mathbb C$ be a mea\-su\-rab\-le function with $|\mu(z)|<1$ a.e.
The equation of the form
\begin{equation}\label{1}
f_{\bar{z}}=\mu(z) f_z\
\end{equation}
where $f_{\bar z}={\bar\partial}f=(f_x+if_y)/2$, $f_{z}=\partial
f=(f_x-if_y)/2$, $z=x+iy$, $f_x$ and $f_y$ are partial derivatives
of the function $f$ in $x$ and $y$, respectively, is said to be a
{\it Beltrami equation}. The  equation~\eqref{1} is said to be {\it
nondegenerate} if $||\mu||_{\infty}<1$, see e.g. \cite{Alf},
\cite{BGMR} and \cite{LV}, that we will  assume later on.

\medskip

Homeomorphic solutions $f$ of a nondegenerate equation (\ref{1}) in
the class $W^{1,2}_{\rm loc}$ are called {\it quasiconformal
mappings}. It is easy to see that every quasiconformal function
$F={\cal A}\circ f$ satisfies the same Beltrami equation as $f.$

\medskip

Recall also that the images of the unit disk $\mathbb D = \{
z\in\mathbb{C}: |z|<1\}$ under the quasiconformal mappings of
$\mathbb C$ onto itself are called {\it quasidisks} and their
boundaries are called {\it quasicircles} or {\it quasiconformal
curves}. It is known that every smooth (or Lipschitz) Jordan curve
is a quasiconformal curve and, at the same time, quasiconformal
curves can be locally nonrectifiable as it follows from the known
examples, see e.g. the point II.8.10 in \cite{LV}. On the other
hand, see Section 3, quasicircles satisfy the well-known
(A)--condition, which is standard in the theory of boundary value
problems for PDE, see e.g. \cite{LU}.

\medskip

Proceeding from the above, the problem under consideration is to
find the quasiconformal function, satisfying both the Beltrami
equation (\ref{1}) in a Jordan domain $D$ and the Hilbert boundary
condition (\ref{2}).  We substantially weaken the regularity
conditions both on the functions $\lambda$ and $\varphi$ in the
boundary condition (\ref{2}) and on the boundary $C$ of the domain
$D.$ On the one hand, we will deal with the coefficients $\lambda$
of {\it countable bounded variation} and the boundary data $\varphi$
which are measurable with respect to {\it the logarithmic capacity}.
On the other hand, the fundamental Becker -- Pommerenke result in
\cite{BP} allows us to study the Hilbert boundary value problem in
domains $D$  with the {\it quasihyperbolic boundary condition}
introduced in \cite{GM}, see also \cite{AK}. It is important to note
that such domains
 may fail to satisfy the (A)--condition, see Section 3.

\medskip

Let $D$ be a Jordan domain such that it has a tangent at a point
$\zeta\in\partial D.$ A path in $D$ terminating at $\zeta$ is
called {\it nontangential} if its part in a neighborhood of
$\zeta$ lies inside of an angle in $D$ with the vertex at $\zeta$.
The limit along all nontangential paths at $\zeta$  is called {\it
angular} at the point. The latter notion is a standard  tool for
the study of the boundary behavior of analytic and harmonic
functions, see e.g., \cite{Du}, \cite{Ku} and \cite{Po}. Further,
the Hilbert boundary condition (\ref{2}) will be understood
precisely in the sense of the angular limit.

\medskip

The notion of the logarithmic capacity is the important tool for our
research, see e.g. \cite{Kar}, \cite{N}, \cite{No}, because the sets
of zero logarithmic capacity are transformed under quasiconformal
mappings into the sets of zero logarithmic ca\-pa\-ci\-ty. Note
that, as it follows from the classic Ahlfors-Beurling example, see
\cite{AB}, the sets of zero length as well as the sets of zero
harmonic measure are not invariant under quasiconformal mappings.
\par
Dealing with measurable boundary date functions $\varphi(\zeta)$
with respect to the logarithmic capacity, we will use the {\it
abbreviation q.e.} ({\it quasi-everywhere}) on a set
$E\subset{\mathbb C},$ if a property holds for all $\zeta\in E$
except its subset of zero logarithmic capacity, see \cite{La}.

\medskip

\section{Definitions and preliminary remarks}


Given a bounded Borel set  $E$ in the plane $\Bbb C$, a {\it mass
distribution} on $E$ is a nonnegative completely additive function
$\nu$ of a set defined on its Borel subsets with $\nu(E)=1$. The
function
\begin{equation}\label{4}
U^{\nu}(z):=\int\limits_{E} \log\left|\frac1{z-\zeta}
\right|\,d\nu(\zeta)
\end{equation}
is called a {\it logarithmic potential} of the mass distribution
$\nu$ at a point $z\in\Bbb C$. A {\it logarithmic capacity} $C(E)$
of the Borel set $E$ is the quantity
\begin{equation}\label{5}
C(E)=e^{-V}\ , \qquad V\ =\ \inf_{\nu}\ V_{\nu}(E)\ , \qquad
V_{\nu}(E)\ =\ \sup_{z}\ U^{\nu}(z)\ .
\end{equation}

It is also well-known the following geometric characterization of
the logarithmic capacity, see e.g. the point 110 in \cite{N}:
\begin{equation}\label{6}
C(E)\ =\ \tau(E)\ :=\ \lim_{n\to\infty}\ V_n^{\frac2{n(n-1)}}
\end{equation}
where $V_n$ denotes the supremum of the product
\begin{equation}\label{7}
V(z_1,\dots,z_n)\ =\prod_{k<l}^{l=1,\dots,n}|z_k-z_l|
\end{equation}
taken over all collections of points $z_1,\dots,z_n$ in the set $E$.
Following F\'ekete, see \cite{F}, the quantity $\tau(E)$ is called
the {\it transfinite diameter} of the set $E$.

\medskip

\begin{remark}\label{R0}
Thus, we see that if $C(E)=0$, then $C(f(E))=0$ for an arbitrary
mapping $f$ that is continuous by H\"older and, in particular, for
quasiconformal mappings on compact sets, see e.g. Theorem II.4.3 in
\cite{LV}.
\end{remark}

\medskip

In order to introduce sets that are measurable with respect to
logarithmic capacity, we define, following \cite{Kar}, {\it inner
$C_*$ and outer $C^*$ capacities: }
\begin{equation}\label{INNER}
C_*(E)\ \colon =\ \sup_{F\subseteq E}\ C(E),\ \ \ \ \ \ \ \
C^{*}(E)\ \colon =\ \inf_{E\subseteq O}\ C(O)
\end{equation}
where supremum is taken over all compact sets $F\subset\Bbb C$ and
infimum is taken over all open sets $O\subset\Bbb C$. A set
$E\subset\Bbb C$ is called {\it measurable with respect to the
logarithmic capacity} if $C^{*}(E) = C_*(E),$ and the common value
of $C_*(E)$ and $C^*(E)$ is still denoted by $C(E)$.

A function $\varphi:E\to\mathbb{C}$ defined on a bounded set
$E\subset\mathbb{C}$ is called {\it measurable with respect to
logarithmic capacity} if, for all open sets $O\subseteq\Bbb C$,  the
sets
\begin{equation}\label{FUN}
\Omega =\{ z\in E: \varphi(z)\in O\}
\end{equation}
are measurable with respect to logarithmic capacity. It is clear
from the definition that the set E is itself measurable with respect
to logarithmic capacity.

\medskip

Note also that sets of logarithmic capacity zero coincide with sets
of the so-called {\it absolute harmonic measure} zero introduced by
Nevanlinna, see Chapter V in \cite{N}. Hence a set $E$ is of
(Hausdorff) length zero if $C(E)=0$, see Theorem V.6.2 in \cite{N}.
However, there exist sets of length zero having a positive
logarithmic capacity, see e.g. Theorem IV.5 in \cite{Kar}.

\medskip

\begin{remark}\label{R1}
It is known that Borel sets and, in particular, compact and open
sets are measurable with respect to logarithmic capacity, see e.g.
Lemma I.1 and Theorem III.7 in \cite{Kar}. Moreover, as it follows
from the definition, any set $E\subset\Bbb C$ of finite logarithmic
capacity can be represented as a union of a sigma-compactum (union
of countable collection of compact sets) and a set of logarithmic
capacity zero. It is also known that the Borel sets and, in
particular, compact sets are measurable with respect to all
Hausdorff's measures and, in particular, with respect to measure of
length, see e.g. theorem II(7.4) in \cite{S}. Consequently, any set
$E\subset\Bbb C$ of finite logarithmic capacity is measurable with
respect to measure of length. Thus, on such a set any function
$\varphi:E\to\mathbb{C}$ being measurable with respect to
logarithmic capacity is also measurable with respect to measure of
length on $E$. However, there exist functions that are measurable
with respect to measure of length but not measurable with respect to
logarithmic capacity, see e.g. Theorem IV.5 in \cite{Kar}.
\end{remark}

\bigskip

We call $\lambda:\partial\mathbb D\to\mathbb C$ a {\it function of
bounded variation}, write $\lambda\in\mathcal{BV}(\partial\mathbb
D)$, if
\begin{equation}\label{20}
V_{\lambda}(\partial\mathbb D)\ \colon =\ \sup\
\sum\limits_{j=1}\limits^{j=k}\
|\lambda(\zeta_{j+1})-\lambda(\zeta_j)| \ <\ \infty \end{equation}
where the supremum is taken over all finite collections of points
$\zeta_j\in\partial\mathbb D$, $j=1,\ldots , k$, with the cyclic
order meaning that $\zeta_j$ lies between $\zeta_{j+1}$ and
$\zeta_{j-1}$ for every $j=1,\ldots , k$. Here we assume that
$\zeta_{k+1}=\zeta_1=\zeta_0$. The quantity
$V_{\lambda}(\partial\mathbb D)$ is called the {\it variation of the
function} $\lambda$.

\bigskip

\begin{remark}\label{R4}
It is clear by the triangle inequality that if we add new
intermediate points in the collection $\zeta_j$, $j=1,\ldots , k$,
then the sum in \eqref{20} does not decrease. Thus, the given
supremum is attained as $\delta=\sup\limits_{j=1,\ldots
k}|\zeta_{j+1}-\zeta_j|\to 0$. Note also that by the definition
$V_{\lambda}(\partial\mathbb D)=V_{\lambda\circ h}(\partial\mathbb
D)$, i.e., the {\it variation is invariant} under every
homeomorphism $h:\partial\mathbb D\to\partial\mathbb D$ and, thus,
the definition can be extended in a natural way to an arbitrary
Jordan curve in $\mathbb C$.
\end{remark}

\bigskip

The following statement was proved as Proposition 5.1 in the paper
\cite{ER} where the function $\alpha_{\lambda}$ has been called by a
{\it function of argument} of $\lambda$.

\bigskip

\begin{proposition}\label{P2} {\it
For every function
$\lambda:\partial\mathbb{D}\to\partial\mathbb{D}$ of the class
$\mathcal{BV}(\partial\mathbb{D})$ there is a function
$\alpha_{\lambda}:\partial\mathbb{D}\to\mathbb{R}$ of the class
$\mathcal{BV}(\partial\mathbb{D})$ with $V_{\alpha_{\lambda}}\le
V_{\lambda}\, 3\pi/2$ such that
$\lambda(\zeta)=\exp\{{i\alpha_{\lambda}}(\zeta)\}$ for all
$\zeta\in\partial\mathbb{D}$.}
\end{proposition}

\bigskip

Now, we call $\lambda:\partial\mathbb D\to\mathbb C$ a function of
{\it countable bounded variation}, write
$\lambda\in\mathcal{CBV}(\partial\mathbb D)$, if there is a
countable collection of mutually disjoint arcs $\gamma_n$ of
$\partial\mathbb D$, $n=1,2,\ldots$ on each of which the restriction
of $\lambda$ is of bounded variation $V_n$ and the set
$\partial\mathbb D\setminus \cup\gamma_n$ has logarithmic capacity
zero. In particular, the latter holds true if $\partial\mathbb
D\setminus \cup\gamma_n$ is countable. Choosing smaller $\gamma_n$,
we may assume that $\sup\limits_n V_n<\infty$. It is clear, such
functions can be singular enough, see e.g. \cite{DMRV}.

\bigskip

The definition is also extended in the natural way to an arbitrary
Jordan curve $\Gamma$ in $\mathbb C$. Later on,
$L_c^{\infty}(\Gamma)$ denotes the class of all functions
$\alpha:\Gamma\to\mathbb R$ which are measurable with respect to
logarithmic capacity such that $\alpha$ is q.e. bounded on
$\Gamma.$

\bigskip

\begin{proposition}\label{P22} {\it
For every function $\lambda:\partial\mathbb{D}\to\partial\mathbb{D}$
in the class $\mathcal{CBV}(\partial\mathbb{D})$ there is a function
$\alpha_{\lambda}:\partial\mathbb{D}\to\mathbb{R}$ in the class
$L_c^{\infty}(\partial\mathbb{D})\cap\mathcal{CBV}(\partial\mathbb{D})$
such that
\begin{equation}\label{ARGUMENT}
\lambda(\zeta)=\exp\{{i\alpha_{\lambda}}(\zeta)\}\ \ \ \ \ \
\mbox{q.e. on $\partial\mathbb{D}$}.
\end{equation}
}
\end{proposition}
\begin{proof}
Denote by $\lambda_n$ the function on $\partial\mathbb{D}$ that is
equal to $\lambda$ on $\gamma_n$ and to $1$ outside of $\gamma_n$.
Let $\alpha_n$ correspond to $\lambda_n$ by Proposition \ref{P2}.
Then its variation $V_n^*\le V_n\, 3\pi/2$. With no loss of
generality we may assume that $\alpha_n\equiv 0$ outside of
$\gamma_n$. Set $\alpha =\sum\limits_{n=1}\limits^{\infty}\alpha_n$.
Then $\alpha\in\mathcal{CBV}(\partial\mathbb{D})$ and
$\lambda(\zeta)=\exp\{{i\alpha}(\zeta)\}$ q.e. on $\partial{\mathbb
D}.$ Applying the corresponding shifts (divisible $2\pi$), we may
change $\alpha_n$ on $\gamma_n$ through $\alpha_n^*$ with
$|\alpha_n^*|\le\pi$ at the middle point of $\gamma_n$. Then it is
clear that the new function
$\alpha^*\in\mathcal{CBV}(\partial\mathbb{D})$ and
$\lambda(\zeta)=\exp\{{i\alpha^*}(\zeta)\}$ q.e. on
$\partial{\mathbb D}$ and, moreover, $|\alpha^*|\le\pi + V_n\,
3\pi/2$ on every $\gamma_n$, i.e. $|\alpha^*|$ is bounded on the set
$\partial\mathbb D\setminus \cup\gamma_n$. In addition, by the
construction, the function $\alpha^*$ is continuous q.e. on
$\partial{\mathbb D}.$  Hence $\alpha^*\in
L_c^{\infty}(\partial\mathbb D)$.
\end{proof}$\Box$

\bigskip

We say that a Jordan curve $\Gamma$ in $\mathbb C$ is {\it almost
smooth} if $\Gamma$ has a tangent quasi--everywhere. Here we say
that a straight line $L$ in $\mathbb C$ is {\it tangent} to
$\Gamma$ at a point $z_0\in\Gamma$ if
\begin{equation} \label{eqTANGENT}
\limsup\limits_{z\to z_0, z\in \Gamma}\ \frac{\hbox{dist}\, (z,
L)}{|z-z_0|}\ =\ 0\ .
\end{equation}
In particular, $\Gamma$ is almost smooth if $\Gamma$ has a tangent
at all its points except a countable set. The nature of such Jordan
curves $\Gamma$ is complicated enough because the countable set can
be everywhere dense in $\Gamma$.

\bigskip

\begin{remark}\label{RBP}
By Corollary of Theorem 1  in \cite{BP}, a conformal mapping of a
Jordan domain $D$ in $\mathbb C$ with the quasihyperbolic boundary
condition, see the definition in Section 3, onto the unit disk
$\mathbb D,$ as well as its inverse are H\"older continuous in the
closure of $D$ and $\mathbb D$, respectively. Thus, by Remark
\ref{R0} these mappings keep the sets of the logarithmic capacity
zero on boundaries of $D$ and $\mathbb D$. Consequently, by Remark
\ref{R1}, such mappings also keep boundary functions which are
measurable with respect to the logarithmic capacity. These facts
are key for the research of the boundary value problems in the
given domains.
\end{remark}

\section{On domains with quasihyperbolic boundary condition}

Let $D$ be a domain in $\mathbb C$. As usual, here $k_D(z,z_0)$
denotes the {\it quasihyperbolic distance} between points $z$ and
$z_0$ in $D$,
\begin{equation} \label{eqHYPERD}
k_D(z,z_0)\ :=\ \inf\limits_{\gamma} \int\limits_{\gamma}
\frac{ds}{d(\zeta,\partial D)}\ ,
\end{equation}
introduced in the paper \cite{GP}, see also the monographs
\cite{AVV} and \cite{Vu}. Here $d(\zeta,\partial D)$ denotes the
Euclidean distance from the point $\zeta\in D$ to $\partial D$ and
the infimum is taken over all rectifiable curves $\gamma$ joining
the points $z$ and $z_0$ in $D$.

Further, it is said that a domain $D$ satisfies the {\it
quasihyperbolic boundary condition} if
\begin{equation} \label{eqHYPERB}
k_D(z,z_0)\ \le\ a\,\ln \frac{d(z_0,\partial D)}{d(z ,\partial
D)}\, +\, b\ \ \ \ \ \ \ \ \ \forall\ z\in D
\end{equation}
for constants $a$ and $b$ and a point $z_0\in D$. The latter notion
was introduced in \cite{GM} but, before it, was first applied in
\cite{BP}.

\bigskip

\begin{remark}\label{RB}
Quasidisks $D$ satisfy the quasihyperbolic boundary condition.
Indeed, as well--known the Riemann conformal mapping $\omega : D\to
\mathbb D$ is extended to a quasiconformal mapping of $\mathbb C$
onto itself, see e.g. Theorem II.8.3 in \cite{LV}. By one of the
main Bojarski results, see \cite{Bo} and \cite{Bo^*}, Theorem 3.5,
the derivatives of quasiconformal mappings in the plane are locally
integrable with some power $q>2$. Note also that its Jacobian
$J(w)=|\omega_w|^2-|\omega_{\bar{w}}|^2$, see e.g. I.A(9) in
\cite{Alf}. Consequently, in this case $J\in L^p(D)$ for some $p>1$
and we have the desired conclusion by the criterion in \cite{AK},
Theorem 2.4.
\end{remark}

\bigskip

Recall that a domain $D$ in ${\mathbb R}^n$, $n\ge 2$, is called
satisfying {\it (A)--condition} if
\begin{equation} \label{eqA}
\hbox{mes}\ D\cap B(\zeta,\rho)\ \le\ \Theta_0\,\hbox{mes}\
B(\zeta,\rho)\ \ \ \ \ \ \ \ \forall\ \zeta\in\partial D\ ,\
\rho\le\rho_0
\end{equation}
for some $\Theta_0$ and $\rho_0\in(0,1)$, see 1.1.3 in \cite{LU}.
Recall also that a domain $D$ in ${\mathbb R}^n$, $n\ge 2$, is said
to be satisfying the {\it outer cone condition} if there is a cone
that makes possible to be touched by its top to every boundary point
of $D$ from the completion of $D$ after its suitable rotations and
shifts. It is clear that the outer cone condition implies
(A)--condition. It is well known that the above conditions are standard in the theory of boundary value
problems for the partial differential equations.

\bigskip

\begin{remark}\label{A}
Note that quasidisks $D$ satisfy (A)--condition. Indeed, the
quasidisks are the so--called $QED-$domains by Gehring--Martio, see
Theorem 2.22 in \cite{GM1}, and the latter satisfy the condition
\begin{equation} \label{eqQ} \hbox{mes}\ D\cap B(\zeta,\rho)\ \ge\
\Theta_*\,\hbox{mes}\ B(\zeta,\rho)\ \ \ \ \ \ \ \ \forall\
\zeta\in\partial D\ ,\ \rho\le\hbox{diam}\, D
\end{equation}
for some $\Theta_*\in(0,1)$, see Lemma 2.13 in \cite{GM1}, and
quasidisks (as domains with quasihyperbolic boundary condition) have
boundaries of the Lebesgue measure zero, see e.g. Theorem 2.4 in
\cite{AK}. Thus, it remains to note that, by definition, the
completions of quasidisks $D$ in the the extended complex plane
$\overline{\mathbb C}:=\mathbb C\cup\{\infty\}$ are also quasidisks
up to the inversion with respect to a circle in $D$. \end{remark}

\bigskip

As we know, the first example of a simply connected plane domain $D$
with the quasihyperbolic boundary condition which is not a quasidisk
was constructed in \cite{BP}, Theorem 2. However, this domain had
(A)--condition.

\bigskip

\begin{remark}\label{RS} Probably one of the simplest examples of a domain $D$
with the quasihyperbolic boundary condition and without
(A)--condition is the union of 3 open disks with the radius 1
centered at the points $0$ and $1\pm i$. It is clear that the domain
has zero interior angle at its boundary point $1$ and by Remark
\ref{A} it is not a quasidisk. Note that $\partial D$ is almost
smooth. Thus, there exist almost smooth Jordan curves with the
quasihyperbolic boundary condition that are not quasiconformal
curves.
\end{remark}

\bigskip

From now on we will naturally assume that the boundary Jordan
curves $\Gamma:=\partial D$ are almost smooth.


\section{Boundary correlation of conjugate harmonic functions}

It is known the very delicate observation due to Lusin that harmonic
functions in the unit circle with continuous (even absolutely
continuous !) boundary data can have conjugate harmonic functions
whose boundary data are not continuous functions, furthemore, they
can be even not essentially bounded in neighborhoods of each point
of the unit circle, see e.g. Theorem VIII.13.1 in \cite{Bari}. Thus,
a correlation between boundary data of conjugate harmonic functions
is not a simple matter, see e.g. I.E in \cite{Ku}, see also
\cite{R4} and \cite{R5}.

\bigskip

The following statement was first proved for the case of bounded
variation in \cite{ER}. Here we give an alternative proof of this
significant fact and extend it to the case of countable bounded
variation.

\medskip

\begin{lemma}\label{T3C} {\it
Let $\alpha:\partial\mathbb D\to\mathbb R$ be in the class
$L_c^{\infty}(\partial\mathbb
D)\cap\mathcal{CBV}(\partial\mathbb{D})$ and let $u:\mathbb
D\to\mathbb R$ be a bounded harmonic function such that
\begin{equation}\label{23C}
\lim\limits_{z\to\zeta}\ u(z)\ =\ \alpha(\zeta)
\end{equation}
at every point of continuity of $\alpha$ and let $v$ be its
conjugate harmonic function. Then $v$ has the angular limit
\begin{equation}\label{24C}
\lim\limits_{z\to\zeta}\ v(z)\ =\ \beta(\zeta)\,\,\,\,\,\,\mbox{
q.e. on ${\partial{\mathbb D}}$},
\end{equation}
where the function $\beta:\partial\mathbb D\to\mathbb R$ is
measurable with respect to the logarithmic capacity.}
\end{lemma}

\bigskip

\begin{proof} Let us start from the case $\alpha\in\mathcal{BV}(\partial\mathbb{D}).$
In this case $\alpha$ has at most a countable set $S$ of points of
discontinuity and, consequently, $S$ is of zero logarithmic
capacity. Hence by the generalized maximum principle, see e.g. the
point 115 in \cite{N}, such a function $u$ is unique and, thus, $u$
can be represented as the Poisson integral of the function $\alpha$,
see e.g. Theorem I.D.2.2 in \cite{Ku},
\begin{equation}\label{PoissonV}
u(re^{i\vartheta})\ =\ \frac{1}{2\pi}\
\int\limits_{-\pi}\limits^{\pi}\frac{1-r^2}{1-2r\cos(\vartheta
-t)+r^2}\ \alpha(e^{it})\ dt\ .
\end{equation}
Here the Poisson kernel is a real part of the analytic function
$(\zeta +z)/(\zeta -z)$, $\zeta = e^{it}$, $z=re^{i\vartheta}$, and
by the Weierstrass theorem, see e.g. Theorem 1.1.1 in \cite{Go}, the
Schwartz integral
\begin{equation}\label{PoissonW}
f(z)\ :=\ \frac{1}{2\pi i}\ \int\limits_{\partial\mathbb D}
\alpha(\zeta)\ \frac{\zeta +z}{\zeta -z}\ \frac{d\zeta}{\zeta}
\end{equation}
gives the analytic function $f=u+iv$ in $\mathbb D$ with  $u={\rm
Re}\, f$, $v={\rm Im}\, f$, and
\begin{equation}\label{PoissonT}
f(z)\ =\ \frac{1}{2\pi}\ \int\limits_{-\pi}\limits^{\pi}
\alpha(e^{it})\ \frac{e^{it} +z}{e^{it} -z}\ dt\ =\ C\ +\
\frac{z}{\pi}\ \int\limits_{-\pi}\limits^{\pi}
\frac{F(t)}{1-e^{-it}z}\ dt
\end{equation}
where $F(t)=e^{-it}\alpha(e^{it})$ and $C=\frac{1}{2\pi}
\int\limits_{-\pi}\limits^{\pi} \alpha(e^{it})\ dt$. By Theorem 2(c)
in \cite{T} the function $f(z)$ has angular limits  $f(\zeta)$ as
$z\to\zeta$ q.e. on $\partial\mathbb D$ because the function $F$ is
of bounded variation. It remains to note that
$f(\zeta)=\lim\limits_{n\to\infty}\, f_n(\zeta)$, where
$f_n(\zeta)=f(r_n\zeta)$, for an arbitrary sequence $r_n\to 1-0$ as
$n\to \infty$ q.e. on $\partial\mathbb D$ and, thus, $f(\zeta)$ is
measurable with respect to logarithmic capacity because the
functions $f_n(\zeta)$ are so as continuous functions on
$\partial\mathbb D$, see e.g. 2.3.10 in \cite{Fe}.

\bigskip

Now, let $\alpha\in\mathcal{CBV}(\partial\mathbb{D})$. Then its set
of points of discontinuity is at most of zero logarithmic capacity.
Hence again by the generalized maximum principle the bounded
function $u$ satisfying (\ref{23C}) is unique. Moreover, $\alpha\in
L_c^{\infty}(\partial\mathbb D)$ and, consequently, $u$ can be
represented by the Poisson integral (\ref{PoissonV}) and the
Schwartz integral (\ref{PoissonW}) gives the analytic function
$f=u+iv$ in $\mathbb D$, where
\begin{equation}\label{PoissonVC}
v(re^{i\vartheta})\ =\ \frac{1}{2\pi}\
\int\limits_{-\pi}\limits^{\pi}\frac{2r\sin(\vartheta-t)}{1-2r\cos(\vartheta-t)+r^2}\
\alpha(e^{it})\ dt\ .
\end{equation}

\bigskip

Let us apply the linearity of the integral operator
(\ref{PoissonVC}). Namely, denote by $\chi$ the characteristic
function of an arc $\gamma_*$ of $\partial\mathbb D$ where $\alpha$
is of bounded variation from the definition of $\mathcal{CBV}$.
Setting $\alpha_*=\alpha\,\chi$ and $\alpha_0=\alpha -\alpha_*$, we
have that $\alpha =\alpha_*+\alpha_0$. Then $v=v_*+v_0$ where $v_*$
and $v_0$ correspond to $\alpha_*$ and $\alpha_0$ by formula
(\ref{PoissonVC}). By the first item of the proof, there exists the
angular limit $\lim\limits_{z\to\zeta}\ v_*(z)\ =\ \beta_*(\zeta)$
q.e. on $\partial\mathbb D$ where $\beta_*:\partial\mathbb
D\to\mathbb R$ is a measurable function with respect to the
logarithmic capacity. Moreover, it is evident from formula
(\ref{PoissonVC}) that $v_0(z)\to \beta_0(\zeta)$ as $z\to\zeta$ for
all $\zeta\in\gamma_*$ where $\beta_0:\gamma_*\to\mathbb R$ is
continuous on $\gamma_*$. Thus, setting $\beta=\beta_*+\beta_0$ on
$\gamma_*$, we obtain the conclusion of Lemma \ref{T3C}, because the
collection of such arcs $\gamma_*$ is  countable  and the completion
of this collection on $\partial\mathbb D$ has zero logarithmic
capacity.
\end{proof}$\Box$

\bigskip

\section{The Hilbert problem for analytic functions in the disk}

\bigskip

Now we are ready to give a solution to the Hilbert boundary value
problem for analytic functions in the unit disk, assuming that the
coefficient $\lambda$ is of countable bounded variation and the
boundary date $\varphi$ is measurable with respect to the
logarithmic capacity.

\bigskip

\begin{theorem}\label{T4} {\it
Let $\lambda:\partial\mathbb D\to\partial\mathbb D$ be in the
class $\mathcal{CBV}(\partial\mathbb{D})$ and
$\varphi:\partial\mathbb D\to\mathbb R$ be measurable with respect
to the logarithmic capacity. Then there is an analytic function
$f:\mathbb D\to\mathbb C$ that has the angular limit
\begin{equation}\label{25}
\lim\limits_{z\to\zeta}\ \mathrm {Re}[\overline{\lambda(\zeta)}
f(z)]\ =\ \varphi(\zeta) \quad\quad\quad\ \ \mbox{q.e. on
$\partial\mathbb D.$}
\end{equation}}
\end{theorem}
\begin{proof}
By Proposition \ref{P22}, the function $\alpha_{\lambda}\in
L_c^{\infty}(\partial\mathbb{D})\cap\mathcal{CBV}(\partial\mathbb{D})$.
Therefore
$$ g(z)\ :=\ \frac{1}{2\pi i}\ \int\limits_{\partial\mathbb
D}\alpha_{\lambda}(\zeta)\ \frac{z+\zeta}{z-\zeta}\
 \frac{d\zeta}{\zeta}\ , \ \ \ \ \ z\in\mathbb D\ ,
$$
is  analytic function with $u(z)={\mathrm Re}\
g(z)\to\alpha_{\lambda}(\zeta)$ as $z\to\zeta$ for every
$\zeta\in\partial\mathbb D$ except a set of the discontinuity
points for the function
  $\alpha_{\lambda},$ which has zero  logarithmic capacity,
see e.g. Corollary IX.1.1 in \cite{Go} and Theorem I.D.2.2 in
\cite{Ku}. Note that the function ${\cal A}(z):=\exp\{ig(z)\}$ is
also analytic.

By Lemma \ref{T3C} there is a function $\beta:\partial\mathbb
D\to\mathbb R$  that has the angular limit $v(z)={\mathrm {Im}}\
g(z)\to\beta(\zeta)$ as $z\to\zeta$  q.e. on $\partial\mathbb D$ and
$\beta$ is measurable with respect to the logarithmic capacity.
Thus, by Corollary 4.1 in \cite{ER} there exists an analytic
function ${\cal B}:\mathbb D\to\mathbb C$ that has the angular limit
$U(z)=\mathrm {Re}\ {\cal B}(z)\to\varphi(\zeta)\,
\exp\{{\beta(\zeta)}\}$ as $z\to\zeta$  q.e. on $\partial\mathbb D.$
Finally, an elementary computation shows that the desired function
has the form $f={\cal A}\,{\cal B}$.
\end{proof}$\Box$

\bigskip

\section{The Hilbert problem for the Beltrami equation}


We say that a function $f:{D}\to{\Bbb C}$ is a {\it regular
solution of the Beltrami equation}~\eqref{1} if $f$ is continuous,
discrete and open, has the first generalized derivatives  and
satisfies~\eqref{1} a.e. in $D.$ We also say that $f$ is a {\it
regular solution of the Hilbert boundary value problem}~\eqref{2}
for the Beltrami equation~\eqref{1} if $f$ in addition
satisfies~\eqref{2}  q.e. on $\partial D$  along nontangential
paths in $D.$

Recall that a mapping $f:{D}\to{\Bbb C}$ is called {\it discrete}
if the pre-image $f^{-1}(z)$ consists of isolated points for every
$z\in \Bbb C$, and {\it open} if $f$ maps every open set
$U\subseteq D$ onto an open set in $\Bbb C$. By the known Stoilow
result, see e.g. \cite{St}, every regular solution $f$ of
(\ref{1}) has the representation $f=h\circ g$ where $g$ is a
homeomorphic solution of (\ref{1}) and $h$ is an analytic
function.

\medskip

\begin{theorem}\label{T6} {\it
Let $D$ be  a Jordan domain with the quasihyperbolic boundary
condition and let $\partial D$ have a tangent q.e. Suppose that
$\mu:D\to\mathbb{C}$ is in $L^\infty(D)$ with $||\mu||_{\infty}<1$,
$\lambda:\partial D\to\mathbb{C},\: |\lambda(\zeta)|\equiv1$, is in
$\mathcal{CBV}(\partial{D})$ and $\varphi:\partial D\to\mathbb{R}$
is a measurable function with respect to the logarithmic capacity.
Then the Hilbert problem \eqref{2} for the Beltrami equation
\eqref{1} has a regular solution. }
\end{theorem}

\medskip



\begin{proof}
Let $g$ be a conformal mapping of $D$ onto $\mathbb D$ that exists
by the Riemann mapping theorem, see e.g. Theorem II.2.1 in
\cite{Go}. Setting in the unit disk $\mathbb D$
\begin{equation}\label{N}
\nu(w)\ :=\
\left(\mu\,\frac{g^{\prime}}{\overline{g^{\prime}}}\right)\circ
g^{-1}(w)\ ,\end{equation} we see that $\nu\in L^\infty(\mathbb
D)$  and $\|\nu\|_{\infty}=\|\mu\|_{\infty}<1$. Hence, by the
Measurable Riemann Mapping theorem, see e.g. \cite{Alf},
\cite{BGMR} and \cite{LV}, there is a quasiconformal mapping $G$
of $\mathbb D$ onto itself, $G(0)=0$, satisfying the Beltrami
equation $G_{\bar w} = \nu(w)\, G_w$ a.e. in $\mathbb D$.

By the reflection principle, see e.g. Theorem I.8.4 in \cite{LV},
$G$ can be extended to a quasiconformal mapping $\tilde{G}$ of
$\mathbb C$ onto itself. Both functions
$G_*:=\tilde{G}|_{\partial\mathbb D}$ and $G_*^{-1}$ are H\"older
continuous, see \cite{Bo^*}, Theorem 3.5, and also \cite{LV},
Theorem II.4.3.

Now, by the Caratheodory theorem, see e.g. Theorem II.3.4 in
\cite{Go}, $g$ is extended to a homeomorphism $\tilde g$ of
$\overline{D}$ onto $\overline{\mathbb D}$. By Corollary of
Theorem 1 in \cite{BP}, $g_*:=\tilde{g}|_{\partial D}$ and its
inverse function are H\"older continuous.

Thus, the mapping $h_*:=G_*\circ g_*:\partial D\to\partial\mathbb D$
and its inverse are also H\"older continuous. In particular, then
$\Lambda :=\lambda\circ h_*^{-1}\in\mathcal{CBV}(\partial{\mathbb
D})$ and $\Phi :=\varphi\circ h_*^{-1}$ is measurable with respect
to logarithmic capacity by Remarks \ref{R0} and \ref{RBP}.

Next, by Theorem \ref{T4} there is an analytic function ${\cal
A}:\mathbb D\to\mathbb C$ that has the angular limit
\begin{equation}\label{25B}
\lim\limits_{\omega\to\eta}\ \mathrm {Re}\
\{\overline{\Lambda(\eta)}\, {\cal A}(\omega)\}\ =\ \Phi(\eta)
\quad\quad\quad \mbox{q.e. on $\partial\mathbb D$}.
\end{equation}

Setting  $h:=G\circ g$, we see, by an elementary computation, see
e.g. (1.C.1) in \cite{Alf}, that $h_z=G_w\circ g(z)\,
g^{\prime}(z)$ and $h_{\bar z}=G_{\bar w}\circ
g(z)\,\overline{g^{\prime}(z)}$ a.e. in $D$, i.e. $h$ is a
quasiconformal mapping of $D$ onto $\mathbb D$ satisfying equation
(\ref{1}) a.e. in $D$.

Let us consider the function $f:={\cal A}\circ h$. Since
$f_z={\cal A}^{\prime}\circ h(z)\, h_z$ and $f_{\bar z}={\cal
A}^{\prime}\circ h(z)\, h_{\bar z}$ a.e. in $D$, we see that $f$
satisfies the equation (\ref{1}). On the other hand, the mapping
$f$ is continuous, open and discrete, and therefore $f$ is the
regular solution of (\ref{1}). It remains to show that $f$
satisfies also the boundary condition (\ref{2}).

Indeed, by the Lindel\"of theorem, see e.g. Theorem II.C.2 in
\cite{Ku}, if $\partial D$ has a tangent at a point $\zeta$, then
$\arg\ [g(\zeta)-g(z)]-\arg\ [\zeta-z]\to\mathrm {const}$ as
$z\to\zeta$. In other words, the images under the conformal
mapping $g$ of sectors in $D$ with a vertex at $\zeta$ is
asymptotically the same as sectors in $\mathbb D$ with a vertex at
$w=g(\zeta)$. Consequently, nontangential paths in $D$ are
transformed under $g$ into nontangential paths in $\mathbb D$ and
inversely q.e. on $\partial D$ and $\partial{\mathbb D},$
respectively, because $D$ is almost smooth and $g_*$ and
$g_*^{-1}$ keep sets of logarithmic capacity zero.

Moreover, it is known that the distortion of angles under a
quasiconformal mapping is bounded, see e.g. \cite{A}, \cite{AG}
and \cite{Ta}. Hence the mapping $\tilde{G}:\mathbb C\to\mathbb C$
and its inverse also transform nontangential paths into
nontangential paths and $G_*$ and $G_*^{-1}$ keep sets of
logarithmic capacity zero. Consequently, $h:D\to\mathbb D$ and
$h^{-1}:\mathbb D\to D$ also transform nontangential paths into
nontangential paths q.e. on $\partial D$ and $\partial{\mathbb
D},$ respectively.   Thus, (\ref{25B}) implies the existence of
the angular limit (\ref{2}) q.e. on $\partial D.$
\end{proof}$\Box$

\medskip

\begin{remark}\label{H1}
 The regular solution $f$ of the Hilbert boundary value problem  for the Beltrami equation given in Theorem \ref{T6}
  has the following representation $f={\cal A}\circ G\circ g.$ Here
$g:D\to\mathbb D$ stands for a conformal mapping, $G:\mathbb
D\to\mathbb D$ is a quasiconformal mapping, normalized by $G(0)=0$
and satisfying the Beltrami equation with the coefficient $\nu$ in
(\ref{N}). Finally,  ${\cal A}:\mathbb D \to \mathbb C$ is the
analytic solution of the Hilbert problem with coefficient $\Lambda
=\lambda\circ h_*^{-1}$ and boundary data $\Phi =\varphi\circ
h_*^{-1}$, where $h=G\circ g$ and $h_*$ is the corresponding
boundary homeomorphism of $\partial D$ onto $\partial{\mathbb D}.$
\end{remark}

\section{On Dirichlet, Neumann and Poincare problems}

We reduce these boundary value problems to suitable Hilbert
problems studied above  and start with the Laplace equation. In
particular, choosing $\mu\equiv 0$ and $\lambda\equiv 1$ in
Theorem \ref{T6}, we immediately obtain the following solution of
the Dirichlet boundary value problem.

\medskip

\begin{corollary}\label{C1}
{\it Let $D$ be  a Jordan domain with the quasihyperbolic boundary
condition and let $\partial D$ have a tangent q.e.  Suppose
$\varphi:\partial D\to\mathbb{R}$ is measurable with respect to the
logarithmic capacity. Then there exists a harmonic function
$u:D\to\mathbb C$ that has the angular limit
\begin{equation}\label{3}
\lim\limits_{z\to\zeta}\ u(z)\ =\ \varphi(\zeta) \quad\quad\quad\ \
\ \ \ \ \ \mbox{
 q.e. on $\partial D.$}
\end{equation}
}
\end{corollary}

We proceed to the study of nonclassical solutions of the Neumann
boundary value problem. For this goal, we will study the more
general {\it problem on directional derivatives}, that in turn is a
partial case of the Poincare boundary value problem.
\par
First of all, let us recall the classical setting of the  problem on
directional derivatives for the Laplace equation in the unit disk
$\mathbb D:$ To find a twice con\-ti\-nu\-ous\-ly differentiable
function $u:\mathbb D\to\mathbb R$ that  admits a continuous
extension to the boundary $\partial{\mathbb D}$ together with its
first partial derivatives, satisfies the Laplace equation
\begin{equation}\label{eqLAPLACE}
\Delta u\ :=\ \frac{\partial^2 u}{\partial x^2}\ +\ \frac{\partial^2
u}{\partial y^2}\ =\ 0 \quad\quad\quad\forall\ z\in\mathbb D
\end{equation}
and the boundary condition
\begin{equation}\label{eqDIRECT}
\frac{\partial u}{\partial \nu}\ =\ \varphi(\zeta) \quad\quad\quad
\forall\ \zeta\in\partial {\mathbb D}.
\end{equation}
Here $\varphi : \partial\mathbb D\to\mathbb R$ stands for a
prescribed continuous function  and $\frac{\partial u}{\partial
\nu}$ denotes the derivative of $u$ at the point $\zeta$ in the
direction $\nu = \nu(\zeta)$, $|\nu(\zeta)|=1$, i.e.,
\begin{equation}\label{eqDERIVATIVE}
\frac{\partial u}{\partial \nu}\ :=\ \lim_{t\to 0}\
\frac{u(\zeta+t\,\nu)-u(\zeta)}{t}\ .
\end{equation}

The  Neumann boundary value problem for the Laplace equation is a
special case of the above problem with the following boundary
condition
\begin{equation}\label{eqNEUMANN}
\frac{\partial u}{\partial n}\ =\ \varphi(\zeta) \quad\quad\quad
\forall\ \zeta\in\partial\mathbb D\ ,
\end{equation}
where $n$ denotes the unit interior normal to $\partial\mathbb D$ at
the point $\zeta$.
\par
Let us note that the above problem on directional derivatives is a
partial case of the {\it Poincare  boundary value problem}
\begin{equation}\label{eqPOINCARE}
a\, u\ +\ b\,\frac{\partial u}{\partial \nu}\ =\ \varphi(\zeta)
\quad\quad\quad \forall\ \zeta\in\partial\mathbb D
\end{equation}
where $a=a(\zeta)$ and $b=b(\zeta)$ are real-valued functions given
on $\partial\mathbb D$.
\par
It is well known, that the Neumann problem, in general, has no
classical solution. The necessary condition for the  solvability
is  that the integral of the function $\varphi$ over
$\partial\mathbb D$ is equal zero, see e.g. \cite{M}. Recently, it
was established the existence of nonclassical solutions of the
Neumann problem  for the Laplace equation in rectifiable Jordan
domains for arbitrary measurable data with respect to the natural
parameter, see \cite{R3}. Then the results have been extended to
linear divergence equations in Lipschitz domains  with arbitrary
measurable data with respect to the logarithmic capacity, see
\cite{Ye}. Here we extend the corresponding results to wider
classes of domains and boundary functions.
\bigskip

\begin{theorem}\label{T1} {\it Let $D$ be  a Jordan domain with the quasihyperbolic boundary
condition and let $\partial D$ have a tangent q.e. Suppose that
$\nu:\partial D\to\mathbb{C},\: |\nu(\zeta)|\equiv 1$, is in the
class $\mathcal{CBV}$ and $\varphi:\partial D\to\mathbb{R}$ is
measurable with respect to the logarithmic capacity. Then there
exists a harmonic function $u: D\to\mathbb{R}$ that has the angular
limit
\begin{equation}\label{6.6}
\lim_{z\to\zeta}\ \frac{\partial u}{\partial\nu}\ =\varphi(\zeta)\ \
\ \ \ \ \mbox{q.e. on\ $\partial D$.}  \end{equation} }
\end{theorem}
{\it Proof.} Indeed, by Theorem \ref{T6} there exists an analytic
function $f:D\to\mathbb{C}$  that has the angular limit
\begin{equation}\label{6.7} \lim_{z\to\zeta}\mathrm{Re}\left[ {\nu(\zeta)}\, f(z)\right]\ =\ {\varphi(\zeta)}
\end{equation}  q.e. on $\partial D.$  Note that an indefinite integral $F$ of $f$ in ${D}$ is also an
analytic function and, correspondingly, the harmonic functions
$u=\mathrm{Re}\, F$ and $v=\mathrm{Im}\, F$ satisfy the
Cauchy-Riemann system  $v_x=-u_y$ и $v_y=u_x$. Hence
$$f\ =\ F'\ =\ F_x\ =\ u_x\ +\ i\, v_x\ =\ u_x\ -\ i\, u_y\ =\
\overline{\nabla u}$$ where $\nabla u=u_x+i\, u_y$ is the gradient
of the function $u$ in the complex form. Thus, (\ref{6.6}) follows
from (\ref{6.7}), i.e. $u$ is the desired harmonic function,
because its directional derivative
$$\frac{\partial u}{\partial\nu}\ =\
{\mathrm Re}\, \overline{\nu}\,\nabla u\ =\ {\mathrm Re}\,
{\nu}\,\overline{\nabla u}\ =\ \langle \nu,\nabla u\rangle$$ is
the scalar product of $\nu$ and the gradient $\nabla u.$ $\Box$

\medskip

\begin{remark}\label{N1}
We are able to say more in the case $\mathrm {Re}[
n\overline{\nu}]>0$ where $n=n(\zeta)$ is the unit interior
normal  at the point
$\zeta\in\partial{D}$. In view of (\ref{6.6}), since the limit
$\varphi(\zeta)$ is finite, there is a finite limit $u(\zeta)$ of
$u(z)$ as $z\to\zeta$ in ${D}$ along the straight line passing
through the point $\zeta$ and being parallel to the vector
$\nu(\zeta).$ Indeed, along this line, for $z$ and $z_0$ that are
close enough to $\zeta$, $$ u(z)\ =\ u(z_0)\ -\
\int\limits_{0}\limits^{1}\ \frac{\partial u}{\partial \nu}\
(z_0+\tau (z-z_0))\ d\tau\ .$$ Thus, at each point with the
condition (\ref{6.6}), there is the directional derivative
$$
\frac{\partial u}{\partial \nu}\ (\zeta)\ :=\ \lim_{t\to 0}\
\frac{u(\zeta+t\,\nu)-u(\zeta)}{t}\ =\ \varphi(\zeta)\ .
$$
\end{remark}

In particular, $\mathrm {Re}[ n\overline{\nu}]=1$ in the case of
the Neumann problem and, thus, we arrive, by Theorem \ref{T1} and
Remark \ref{N1}, at the following  result.

\bigskip

\begin{corollary}\label{C2}
{\it Let $D$ be a Jordan domain in $\Bbb C$ with the quasihyperbolic
boundary condition and let the unit interior normal $n(\zeta)$ to
the boundary $\partial D$ be in the class $\mathcal{CBV}$. Suppose
that $\varphi:\partial D\to\mathbb{R}$ is measurable with respect to
the logarithmic capacity. Then one can find a harmonic function
$u:D\to\mathbb C$ such that  q.e. on $\partial D$  there exist:

\bigskip

1) the finite limit along the normal $n(\zeta)$
$$
u(\zeta)\ :=\ \lim\limits_{z\to\zeta}\ u(z)\ ,$$

2) the normal derivative
$$
\frac{\partial u}{\partial n}\, (\zeta)\ :=\ \lim_{t\to 0}\
\frac{u(\zeta+t\, n)-u(\zeta)}{t}\ =\ \varphi(\zeta)\ ,
$$

3) the angular limit
$$ \lim_{z\to\zeta}\ \frac{\partial u}{\partial n}\, (z)\ =\
\frac{\partial u}{\partial n}\, (\zeta)\ .$$}
\end{corollary}

Recall that, see e.g. Theorem 16.1.6 in \cite{AIM}, if $f=u+i\, v$
is a regular solution of the Beltrami equation (\ref{1}), then the
function $u$ is a continuous generalized solution of the
divergence type equation \begin{equation}\label{6.8} {\rm div}\,
A(z)\nabla\,u=0\ , \end{equation} called {\it A-harmonic
function}, see \cite{HKM}, i.e. $u\in C\cap W^{1,1}_{\rm loc}(D)$
and
$$
\int\limits_D \langle A(z)\nabla
u,\nabla\varphi\rangle=0\,\,\,\,\,\,\,\,\,\,\,\,\forall\ \varphi\in
C_0^\infty(D)\ , $$ where $A(z)$ is the matrix function:
\begin{equation}\label{6.9}
A=\left(\begin{array}{ccc} {|1-\mu|^2\over 1-|\mu|^2}  & {-2{\rm Im}\,\mu\over 1-|\mu|^2} \\
                            {-2{\rm Im}\,\mu\over 1-|\mu|^2}          & {|1+\mu|^2\over 1-|\mu|^2}  \end{array}\right).
\end{equation}
As we see, the matrix function $A(z)$ in (\ref{6.9}) is symmetric and its
entries $a_{ij}=a_{ij}(z)$ are dominated by the quantity
$$
K_{\mu}(z)\ =\ \frac{1\ +\ |\mu(z)|}{1\ -\ |\mu(z)|}\ ,
$$
and, thus, they are bounded if Beltrami's equation (\ref{1}) is not
de\-ge\-ne\-ra\-te.

\bigskip

Vice verse, uniformly elliptic equations (\ref{6.8}) with symmetric
$A(z)$ and ${\rm det}\,A(z)\equiv 1$ just correspond to
nondegenerate Beltrami equations (\ref{1}) with coefficient
\begin{equation}\label{6.10}
\mu\ =\ \frac{1}{\mathrm{det}\, (I+A)}\ (a_{22}-a_{11}\ -\
2ia_{21})\ =\ \frac{a_{22}-a_{11}\ -\ 2ia_{21}}{1\ +\ \mathrm{Tr}\,
A\ +\ \mathrm{det}\, A}\ .\end{equation} Following \cite{GRY}, we
denote by ${\cal{B}}$ the collection of all such matrix functions
$A(z).$  Recall that the equation (\ref{6.8}) is said to be {\it
uniformly elliptic}, if $a_{ij}\in L^{\infty}$ and $\langle
A(z)\eta,\eta \rangle\geq\varepsilon|\eta|^2$ for some
$\varepsilon>0$ and for all $\eta\in\mathbb{R}^2$.

\bigskip

\begin{corollary}\label{C3}
{\it Let $D$ be a domain with the qua\-si\-hy\-per\-bo\-lic boundary
condition and let $\partial D$ have a tangent q.e. Suppose that
$A\in\mathcal{B}$ and $\varphi:\partial D\to\mathbb{R}$ is
measurable with respect to logarithmic capacity. Then there exists
$A-$harmonic function $u:D\to\mathbb{R}$ satisfying the Dirichlet
boundary condition (\ref{3}).}
\end{corollary}

\bigskip

\begin{theorem}\label{T2} {\it Let $D$ be a domain in
$\mathbb{C}$ with the quasihyperbolic boundary condition and let
$\partial D$ have a tangent q.e. Suppose that $A(z), \:z\in D$, is a
matrix function in the class $\mathcal{B}\cap
C^{\alpha},\:\alpha\in(0,1)$, $\nu:\partial D\to~\mathbb{C}$,
$|\nu(\zeta)|\equiv1$, is in the class ${\cal CBV}$ and
$\varphi:\partial D\to\mathbb{R}$ is measurable with respect to
logarithmic capacity.

Then there exists $A-$harmonic function $u:D\to\mathbb{R}$ in the
class $C^{1+\alpha}$ that has the angular limit
\begin{equation}\label{6.11} \lim_{z\to\zeta}\ \frac{\partial
u}{\partial\nu}\, (z)\ =\ \varphi(\zeta)\ \ \ \mbox{q.e. on
$\partial D$}\ .
\end{equation}}
\end{theorem}

{\it Proof.} By the above remarks, a desired function $u$ is a real
part of a solution $f$ in class $W^{1,1}_{\mathrm loc}$ for the
Beltrami equation (\ref{1}) with $\mu\in C^{\alpha}$ given by the
formula (\ref{6.10}). By Lemma 1 in \cite{GRY} $\mu$ is extended to
a H\"older continuous function $\mu_*:\mathbb{C}\to\mathbb{C}$ of
the class $C^{\alpha}$. Set $k=\max\, |\mu(z)|<1$ in $D$. Then, for
every $k_*\in(k,1)$, there is an open neighborhood $U$ of $D$ where
$|\mu_*(z)|\le k_*$. Let $D_*$ be a connected component of $U$
containing $\overline D$.

By the Measurable Riemann Mapping Theorem, see e.g. \cite{Alf},
\cite{BGMR} and \cite{LV}, there is a quasiconformal mapping
$h:{{D_*}}\to{\mathbb{C}}$ a.e. satisfying the Beltrami equation
(\ref{1}) with the complex coefficient $\mu^*:=\mu_*|_{D_*}$ in
$D_*$. Note that the mapping $h$ has the H\"older continuous first
partial derivatives in $D_*$ with the same order of the H\"older
continuity as $\mu$, see e.g. \cite{Iw} and also \cite{IwDis}.
Moreover, its Jacobian \begin{equation}\label{6.13} J_h(z)\ne 0\ \ \
\ \ \ \ \ \ \ \ \ \ \forall\ z\in D_*\ , \end{equation} see e.g.
Theorem V.7.1 in \cite{LV}. Thus, the directional derivative
$$h_{\omega}(z)=\frac{\partial h}{\partial\omega}\, (z)\ :=\
\lim_{t\to0}\ \frac{h(z\ +\ t\,\omega)\ -\ h(z)}{t}\ \ne\ 0 \ \ \
\ \ \ \ \ \ \ \ \ \forall\ z\in D_*\ \forall\
\omega\in\partial\mathbb D$$ and it is continuous by the
collection of the variables $\omega\in\partial\mathbb D$ and $z\in
D_*$. Thus, the functions
$$ \nu_*(\zeta)\ :=\
\frac{|h_{\nu(\zeta)}(\zeta)|}{h_{\nu(\zeta)}(\zeta)}\ \ \ \ and \ \
\ \ \varphi_*(\zeta)\ :=\ \frac{\varphi(\zeta)}
{|h_{\nu(\zeta)}(\zeta)|}
$$ are measurable with respect to the logarithmic capacity, see e.g.
convergence arguments in \cite{KZPS}, Section 17.1.

The logarithmic capacity of a set coincides with its transfinite
diameter, see e.g.  \cite{F} and the point 110 in \cite{N}.
Moreover, quasiconformal mappings are H\"{o}lder continuous on
compacta, see e.g. Theorem II.4.3 in \cite{LV}. Hence the mappings
$h$ and $h^{-1}$ transform sets of logarithmic capacity zero on
$\partial D$ into sets of logarithmic capacity zero on $\partial
D^*$, where $D^*:=h(D)$, and vice versa.

\medskip

Further, the functions ${\cal{N}} := \nu_*\circ h^{-1}|_{\partial
D^*}$ and $\Phi :=\left(\varphi_* /h_{\nu}\right)\circ
h^{-1}|_{\partial D^*}$ are mea\-su\-rab\-le with respect to the
logarithmic capacity. Indeed, a measurable set with respect to the
lo\-ga\-rith\-mic capacity is transformed under the mappings $h$ and
$h^{-1}$ into measurable sets with respect to the logarithmic
capacity. Really, such a set can be represented as the union of a
sigma-com\-pac\-tum and a set of logarithmic capacity zero. On the
other hand, the compacta are transformed  under continuous mappings
into compacta and the compacta are measurable with respect to  the
logarithmic capacity.

\medskip

Recall that the distortion of angles under quasiconformal mappings
$h$ и $h^{-1}$ is bounded, see e.g. \cite{A}, \cite{AG} and
\cite{Ta}. Thus, nontangential paths to $\partial D$ are transformed
into nontangential paths to $\partial D^*$ for a.e.
$\zeta\in\partial D$ with respect to the logarithmic capacity and
inversely.

\medskip

By Theorem \ref{T1}, one can find a harmonic function
$U:D^*\to\mathbb{R}$ that has the angular limit
\begin{equation}\label{6.13}\lim_{w\to\xi}\ \frac{\partial
U}{\partial\mathcal{N}}\, (w)\ =\ \Phi(\xi)\ \ \ \ \ \mbox{q.e. on
$\partial D^*$}\ .
\end{equation}

Moreover, one can find a harmonic function $V$ in the simply
connected domain $D^*$ such that $F=U+iV$ is an analytic function
and, thus, $u:=\mathrm{Re}\, f=U\circ h$, where $f:=F\circ h$, is a
desired $A$-harmonic function in Theorem \ref{T2} because $f$ is a
regular solution of the corresponding Beltrami equation (\ref{1})
and also
$$ u_{\nu} = \langle\ \nabla U \circ h\ ,\ h_{\nu}\ \rangle =
\langle\ \nu_*\,\nabla U \circ h\ ,\ \nu_*\, \, h_{\nu}\ \rangle =$$
$$=\langle\ \frac{\partial U}{\partial\mathcal{N}}\ \circ h\ ,\
\nu_*\,\, h_{\nu}\ \rangle = \frac{\partial U}{\partial\mathcal{N}}
\ \circ h\ \,\ {\mathrm Re}\, (\nu_*h_{\nu}).\ \ \ \Box
$$

\bigskip

The following statement concerning to the Neumann problem for
$A$-harmonic functions is a partial case of Theorem \ref{T2}.

\bigskip

\begin{corollary}\label{C4} {\it Let $D$ be a domain in
$\mathbb{C}$  with the quasihyperbolic boundary condition and let
$\partial D$ have a tangent q.e. Suppose that $A(z), \:z\in D$, is a
matrix function in the class $\mathcal{B}\cap
C^{\alpha},\:\alpha\in(0,1)$, the interior unit normal $n=n(\zeta)$
to $\partial{D}$ is in the class ${\cal CBV}$ and $\varphi:\partial
D\to\mathbb{R}$ is measurable with respect to the logarithmic
capacity.

\medskip

Then there is $A-$harmonic function $u:D\to\mathbb{R}$ of the class
$C^{1+\alpha}$ such that q.e. on $\partial D$ there exist:

\medskip

1) the finite limit along the normal $n(\zeta)$
$$
u(\zeta)\ :=\ \lim_{z\to\zeta}\ u(z)$$

2) the normal derivative
$$
\frac{\partial u}{\partial n}\, (\zeta)\ :=\ \lim_{t\to0}\
\frac{u(\zeta+t\, n)\ -\ u(\zeta)}{t}\ =\ \varphi(\zeta)
$$

3) the angular limit
$$ \lim_{z\to\zeta}\ \frac{\partial u}{\partial n}\, (z)\ =\
\frac{\partial u}{\partial n}\, (\zeta)\ .$$}
\end{corollary}

\bigskip

\section{On the dimension of the spaces of solutions}

It was established in \cite{ER}, Theorem 8.1, that the space of all
harmonic functions $u:\mathbb D\to\mathbb R$ that has the angular
limit $\lim\limits_{z\to\zeta}u(z) = 0$ q.e. on $\partial\mathbb D$
has the infinite dimension. This statement can be extended to the
Hilbert boundary value problem because we reduced this problem in
Theorem \ref{T4} to the corresponding two Dirichlet problems.

\medskip

\begin{theorem}\label{T8} {\it
Let $\lambda:\partial\mathbb D\to\partial\mathbb D$ be in class
${\cal{CBV}}(\partial\mathbb D)$ and $\varphi:\partial\mathbb
D\to\mathbb R$ be measurable with respect to logarithmic capacity.
Then the space of all analytic functions $f:\mathbb D\to\mathbb C$
with the angular limit
\begin{equation}\label{25}
\lim\limits_{z\to\zeta}\ \mathrm {Re}\ \{\overline{\lambda(\zeta)}\,
f(z)\}\ =\ \varphi(\zeta)\ \ \ \ \ \mbox{q.e. on $\partial\mathbb
D$}
\end{equation}
has the infinite dimension.}
\end{theorem}

\bigskip

\begin{proof} Let $u:\mathbb D\to\mathbb R$ be a harmonic function
that has the angular limit $0$ q.e. on $\partial\mathbb D$ from
Theorem 8.1 in \cite{ER}. Then there is the unique harmonic function
$v:\mathbb D\to\mathbb R$ with $v(0)=0,$ such that ${\mathcal
C}=u+iv$ is an analytic function. Thus, setting in the proof of
Theorem \ref{T4} $g={\mathcal A}({\mathcal B}+{\mathcal C})$ instead
of $f={\mathcal A}\,{\mathcal B}$, we obtain by Theorem 8.1 in
\cite{ER} that the space of solutions of the Hilbert boundary value
problem (\ref{25}) for analytic functions in Theorem \ref{T4} has
the infinite dimension.
\end{proof} $\Box$

\bigskip

Finally, since the proof of the rest of theorems and corollaries was
sequentially reduced to Theorem \ref{T4}, we come by Theorem
\ref{T8} to the following conclusion.

\medskip

\begin{corollary}\label{C5}
{\it All the spaces of solutions of the boundary value problems in
Theorems \ref{T6}, \ref{T1}, \ref{T2}, Corollaries
\ref{C1}--\ref{C4} also have the infinite dimension.}
\end{corollary}

\medskip

Recently\ it\ was\ established\ by\ us\ a\ number of\ effective\
criteria\ for\ the\ \ exi\-sten\-ce of solutions for the degenerate
Beltrami equations, see e.g. \cite{GRSY}. That makes possible to
consider the boundary value problems for such equations, too.
However, the latter will demand a more deep study of properties of
the mappings with finite distortion, see e.g. the monographs
\cite{HK} and \cite{MRSY}.

\bigskip

{\bf ACKNOWLEDGMENTS.} This work was partially supported by grants
of Ministry of Education and Science of Ukraine, project number is
0119U100421.

\medskip
\noindent
{\bf Vladimir Gutlyanskii, Vladimir Ryazanov\\ and  Artyem Yefimushkin,}\\
Institute of Applied Mathematics and Mechanics,\\
National Academy of Sciences of Ukraine,\\
vgutlyanskii@gmail.com, Ryazanov@nas.gov.ua,\\
a.yefimushkin@gmail.com

\medskip\noindent
{\bf Vladimir Ryazanov,}\\
National University of Cherkasy, Ukraine,\\
Physics Department, Laboratory of Mathematical Physics,\\
vl.ryazanov1@gmail.com

\medskip
\noindent {\bf Eduard Yakubov}\\
Holon Institute of Technology, Holon, Israel,\\
yakubov@hit.ac.il, eduardyakubov@gmail.com


\medskip


\begin{thebibliography}{100}

\bibitem{A} {\it Agard S.} Angles and quasiconformal
mappings in space // J. Anal Math. -- 1969. -- {\bf 22}. -- P.
177–200.

\bibitem{AG} {\it Agard S. B., Gehring F.W.} Angles and quasiconformal
mappings // Proc. London Math. Soc. (3) -- 1965. -- {\bf 14a}. -- P.
1–21.

\bibitem{Alf}
{\it Ahlfors L.} Lectures on Quasiconformal Mappings. -- New York:
Van Nostrand, 1966.

\bibitem{AB}
{\it Ahlfors L., Beurling A.} The boundary correspondence under
quasiconformal mappings // Acta Math. -- 1956. -- {\bf 96}. -- P.
125--142.

\bibitem{AVV} {\it Anderson G.D., Vamanamurthy M.K., Vuorinen M.K.} Conformal invariants,
inequalities and quasiconformal maps. --  New York: Wiley-Intersci.
Publ., 1997.

\bibitem{AIM} {\it Astala K., Iwaniec T., Martin G.J.} Elliptic differential
equations and quasiconformal mappings in the plane. -- Princeton
Math. Ser.  {\bf 48}. --  Princeton: Princeton Univ. Press, 2009.

\bibitem{AK}
{\it Astala K., Koskela P.} {Quasiconformal mappings and global
integrability of the derivative} // J. Anal. Math. -- 1991. --  {\bf
57}. -- P. 203--220.

\bibitem{Bari}
{\it Bari N.K.} Trigonometric series. -- Moscow: Gos. Izd.
Fiz.--Mat. Lit., 1961 [in Russian]; transl. as A treatise on
trigonometric series, Vols I and II. -- New York: Macmillan Co.,
1964.

\bibitem{BP}
{\it Becker J., Pommerenke Ch.} H\"older continuity of conformal
mappings and nonquasiconformal Jordan curves // Comment. Math. Helv.
-- 1982. -- {\bf 57}, no. 2. -- P. 221–225.

\bibitem{Be}
{\it Begehr H.} {Complex analytic methods for partial differential
equations. An introductory text.} -- River Edge, NJ: World
Scientific Publishing Co., Inc., 1994.

\bibitem{BW}
{\it Begehr H., Wen G.Ch.} {Nonlinear elliptic boundary value
problems and their applications.} -- Pitman Monographs and Surveys
in Pure and Applied Mathematics {\bf 80}. -- Harlow: Longman, 1996.

\bibitem{Bo}
{\it Bojarski B.V.} {Homeomorphic solutions of Beltrami systems} (in
Russian) // Dokl. Akad. Nauk SSSR (N.S.) -- 1955. -- {\bf 102}. --
P. 661–664.

\bibitem{Bo^*}
{\it Bojarski B.V.} {Generalized solutions of a system of
differential equations of the first order and elliptic type with
discontinuous coefficients} // Report Dept. Math. Stat. -- 2009. --
{\bf 118}. -- Jyv\"askyl\"a: Univ. of Jyv\"askyl\"a; transl. from
Mat. Sb. (N.S.) -- 1957. -- {\bf 43(85)}. -- P. 451–503.

\bibitem{BGMR}
{\it Bojarski B., Gutlyanskii V., Martio O., Ryazanov V.}, {
Infinitesimal geo\-me\-try of quasiconformal and bi-lipschitz
mappings in the plane}. -- EMS Tracts in Ma\-the\-ma\-tics {\bf 19}.
-- Z\"urich: European Mathematical Society, 2013.

\bibitem{Kar}
{\it Carleson L.} Selected Problems on Exceptional Sets. --
Princeton etc.: Van Nostrand Co., Inc., 1971.

\bibitem{DMRV}
{\it Dovgoshey O., Martio O., Ryazanov V., Vuorinen M.} The Cantor
function // Expo. Math. -- 2006. -- {\bf 24}, no. 1. -- P. 1–37.

\bibitem{Du}
{\it Duren P.L.} { Theory of Hp spaces}. -- {Pure and Applied
Mathematics}, {\bf 38}. -- New York--London: Academic Press,  1970.

\bibitem{ER}
{\it Efimushkin A., Ryazanov V.} On the Riemann-Hilbert problem for
the Beltrami equations. Complex analysis and dynamical systems VI.
Part 2, 299–316, Contemp. Math., 667, Israel Math. Conf. Proc.,
Amer. Math. Soc., Providence, RI, 2016.

\bibitem{Fe}
{\it Federer H.} Geometric Measure Theory. -- Berlin:
Springer-Verlag, 1969.


\bibitem{F} {\it F\'ekete M.} \"Uber die Verteilung der Wurzeln bei gewissen
algebraischen Gleichungen mit ganzzahligen Koeffizienten // Math. Z.
-- 1923. -- {\bf 17}. -- P. 228-249.

\bibitem{G}
{\it Gakhov F.D.} {Boundary value problems}. --  New York: Dover
Publications. Inc., 1990.

\bibitem{GM1}
{\it Gehring F.W., Martio O.} {Quasiextremal distance domains and
extension of quasiconformal mappings} // J. Analyse Math. -- 1985.
-- {\bf 45}. -- P. 181–206.

\bibitem{GM}
{\it Gehring F. W., Martio O.} Lipschitz classes and quasiconformal
mappings // Ann. Acad. Sci. Fenn. Ser. A I Math. -- 1985. -- {\bf
10}. -- P. 203–219.

\bibitem{GP}
{\it Gehring F.W., Palka B.P.} Quasiconformally homogeneous domains
// J. Analyse Math. -- 1976. -- {\bf 30}. -- P. 172–199.

\bibitem{Go}
{\it Goluzin G. M.} {Geometric theory of functions of a complex
variable}. -- Transl. of Math. Monographs, {\bf 26}. -- Providence,
R.I.: American Mathematical Society, 1969.

\bibitem{GR}
{\it Gutlyanskii V., Ryazanov V.}  On recent advances in boundary
value problems in the plane // Ukr. Mat. Visn. -- 2016. --  {\bf
13}, no. 2. -- P. 167–212.

\bibitem{GRSY}
{\it Gutlyanskii V., Ryazanov V., Srebro U., Yakubov E.} {The
Beltrami Equation: A Geometric Approach}. -- Developments in
Mathematics, {\bf 26}. -- New York etc.: Springer, 2012.

\bibitem{GRY}
{\it Gutlyanskii V., Ryazanov V., Yefimushkin A.} On the boundary
value problems for quasiconformal functions in the plane // Ukr.
Mat. Visn. -- 2015. -- {\bf 12}, no. 3. -- P. 363--389; transl. in
J. Math. Sci. (N.Y.) -- 2016. -- {\bf 214}, no. 2. -- P. 200–219.


\bibitem{HKM}
{\it Heinonen J., Kilpel\"ainen T., Martio O.} Nonlinear potential
theory of degenerate elliptic equations. -- Oxford Mathematical
Monographs. -- Oxford Science Publications. The Clarendon Press,
Oxford University Press, New York, 1993.

\bibitem{HK}
{\it Hencl S., Koskela P.} Lectures on mappings of finite
distortion. -- Lecture Notes in Mathematics, {\bf 2096}. -- Cham:
Springer, 2014.

\bibitem{H1}
{\it Hilbert D.} \"Uber eine Anwendung der Integralgleichungen auf
eine Problem der Funktionentheorie. -- Verhandl. des III Int. Math.
Kongr., Heidelberg, 1904.

\bibitem{Iw}
{\it Iwaniec T.}   Regularity of solutions of certain degenerate
elliptic systems of equations that realize quasiconformal mappings
in n-dimensional space// {Differential and integral equations.
Boundary value problems}. - 1979. - p. 97--111. -- Tbilisi: Tbilis.
Gos. Univ.


\bibitem{IwDis}
{\it Iwaniec T.}   Regularity theorems for solutions of partial
differential equations for quasiconformal mappings in several
dimensions// {Dissertationes Math. (Rozprawy Mat.)} - 1982. - {\bf
198}. - 45 pp.

\bibitem{Ku}
{\it Koosis P.} {Introduction to $H_p$ spaces}. -- Cambridge Tracts
in Mathematics, {\bf 115}. --  Cambridge: Cambridge Univ. Press,
1998.

\bibitem{KZPS}
{\it Krasnosel'skii M.A., Zabreiko  P.P., Pustyl'nik  E.I.,
Sobolevskii P.E.}  Integral operators in spaces of summable
functions. -- {Monographs and Textbooks on Mechanics of Solids and
Fluids, Mechanics: Analysis}. -- Leiden: Noordhoff International
Publishing, 1976.

\bibitem{LU}
{\it Ladyzhenskaya O.A., Ural'tseva N.N.} { Linear and quasilinear
elliptic equations}. -- New York-London: Academic Press, 1968;
transl. from {Lineinye i kvazilineinye uravneniya ellipticheskogo
tipa}. -- Moscow: Nauka, 1964.

\bibitem{La}
{\it Landkof N.S.} Foundations of modern potential theory. -- Die
Grundlehren der mathematischen Wissenschaften,  {\bf 180}. -- New
York-Heidelberg: Springer-Verlag, 1972.

\bibitem{LV}
{\it Lehto O., Virtanen K.J.} Quasiconformal mappings in the plane.
-- Berlin, Heidelberg: Springer-Verlag, 1973.

\bibitem{MRSY} {\it Martio O., Ryazanov V., Srebro U., Yakubov E.} Moduli in
Modern Mapping Theory. -- Springer Monographs in Mathematics. -- New
York etc.: Springer, 2009.

\bibitem{M}
{\it Mikhlin S.G.}  {Partielle Differentialgleichungen in der
mathematischen Physik}. --  Math. Lehrb\"ucher und Monographien,
\textbf{30}. -- Berlin: Akademie-Verlag, 1978.

\bibitem{Mus} {\it Muskhelishvili N.I.} {Singular integral equations.
Boundary problems of function theory and their application to
mathematical physics}. -- New York: Dover Publications. Inc., 1992.

\bibitem{N}
{\it Nevanlinna R.} {Eindeutige analytische Funktionen}. --
Michigan: Ann Arbor, 1944.

\bibitem{No}
{\it Noshiro K.} {Cluster sets}. --  Berlin etc.: Springer-Verlag,
1960.

\bibitem{Ta} {\it Taari O.} Charakterisierung der Quasikonformit\"at mit Hilfe
der Winkelverzerrung [German] // Ann. Acad. Sci. Fenn. Ser. A I. --
1966. -- {\bf 390}. -- P. 1--43.

\bibitem{TO}
{\it Trogdon Th., Olver Sh.} {Riemann-Hilbert problems, their
numerical solution, and the computation of nonlinear special
functions.} -- Philadelphia: Society for Industrial and Applied
Mathematics (SIAM), 2016.

\bibitem{Po}
{\it Pommerenke Ch.} { Boundary behaviour of conformal maps}. --
{Grundlehren der Mathematischen Wissenschaften}, {\bf 299}. --
Berlin: Springer-Verlag, 1992.

\bibitem{R1}
{\it Ryazanov V.I.} On the Riemann-Hilbert problem without Index //
Ann. Univ. Bucharest, Ser. Math. - 2014. - {\bf 5 (LXIII)}, no. 1. -
P. 169-178.

\bibitem{R2}
{\it Ryazanov V.} Infinite dimension of solutions of the Dirichlet
problem // Open Math. (the former Central European J. Math.) --
2015. -- {\bf 13}. -- P. 348–350.

\bibitem{R3}
{\it Ryazanov V.} On Neumann and Poincare problems for Laplace
equation // Analysis and Mathematical Physics. -- 2017. -- {\bf 7},
no. 3. -- P. 285–289.

\bibitem{R4}
{\it Ryazanov V.} On the boundary behavior of conjugate harmonic
functions // Proceedings of Inst. Appl. Math. Mech. of Nat. Acad.
Sci. of Ukraine. -- 2017. -- {\bf 31}. -- P. 117–123; see also
arXiv.org: 1710.00323v3 [math.CV] 3 Mar 2018, 8 pp.

\bibitem{R5}
{\it Ryazanov V.} The Stieltjes integrals in the theory of harmonic
functions // Zap. Nauchn. Sem. S.-Peterburg. Otdel. Mat. Inst.
Steklova (POMI). -- 2018. -- {\bf 467}. -- Issledovaniya po Lineinym
Ope\-ra\-to\-ram i Teorii Funktsii. -- 2018. -- {\bf 46}. -- P.
151--168; transl. in J. Math. Sci. (N.Y.) -- 2019. -- {\bf 243}, no.
6. -- P. 922--933.

\bibitem{S}
{\it Saks S.} {Theory of the integral}. -- New York: Dover
Publications Inc., 1964.

\bibitem{St}
{\it Stoilow, S.} {Lektsii o topologicheskikh printsipakh teorii
analiticheskikh funktsii.} (Russian) [Lectures on topological
principles of the theory of analytic functions] Transl. from the
French. -- Moscow: Nauka, 1964.

\bibitem{T}
{\it Twomey J.B.} Tangential boundary behaviour of the Cauchy
integral // J. London Math. Soc. (2). -- 1988. -- {\bf 37}, no. 3.
-- P. 447--454.

\bibitem{Vek}
{\it Vekua I.N.} {Generalized analytic functions}. --  London etc.:
Pergamon Press, 1962.

\bibitem{Vu}
{\it Vuorinen M.} Conformal geometry and quasiregular mappings. --
Lecture Notes in Mathematics. {\bf 1319}. -- Berlin:
Springer-Verlag, 1988.

\bibitem{Ye}
{\it Yefimushkin A.} On Neumann and Poincare Problems in A-harmonic
Analysis // Advances in Analysis. -- 2016. -- {\bf 1}, no. 2. -- P.
114--120.

\end{thebibliography}
\end{document}